\documentclass[11pt,leqno,amscd,pstricks,verbatim]{amsart}
\usepackage{amsmath,amsthm, amssymb, latexsym,}
\usepackage{cases}
\usepackage{mathrsfs}
\usepackage{pifont}
\usepackage{amsfonts}
\usepackage{stmaryrd}
\usepackage{color}

\usepackage[english]{babel}
\usepackage[latin1]{inputenc}
\usepackage[numbers,sort&compress]{natbib}
\usepackage[all]{xy}

\usepackage{hyperref}
\usepackage[T1]{fontenc}

\usepackage{ dsfont}

\usepackage{graphicx,xcolor}

\newtheorem{theorem}{Theorem}[section]
\newtheorem{lemma}[theorem]{Lemma}
\newtheorem{corollary}[theorem]{Corollary}
\newtheorem{proposition}[theorem]{Proposition}

\theoremstyle{definition}

\theoremstyle{remark}
\newtheorem{remark}[theorem]{Remark}

\numberwithin{equation}{section}

\newcommand{\on}[1]{\operatorname{#1}}
\newcommand{\abs}[1]{\lvert#1\rvert}

\newcommand{\un}[1]{\underline{#1}}

\newcommand{\mf}[1]{\mathfrak{#1}} 
\newcommand{\xg}{\backslash} 

\newcommand{\lieg}{\mathfrak{g}} 
\newcommand{\lieh}{\mathfrak{h}} 
\newcommand{\liel}{\mathfrak{l}} 
\newcommand{\liege}{\mathfrak{g}_{\bar{0}}} 
\newcommand{\liego}{\mathfrak{g}_{\bar{1}}} 
\newcommand{\lieb}{\mathfrak{b}}
\newcommand{\lien}{\mathfrak{n}}

\newcommand{\ds}{\oplus} 
\newcommand{\ten}{\otimes} 

\newcommand{\fc}{\mathbb{C}} 
\newcommand{\fz}{\mathbb{Z}} 
\newcommand{\fk}{\mathbb{K}} 
\newcommand{\bb}[1]{\mathbb{#1}} 

\newcommand{\ra}{\rightarrow} 
\newcommand{\sur}{\twoheadrightarrow} 
\newcommand{\mt}{\mapsto} 

\newcommand{\iso}{\simeq} 

\newcommand{\pr}[1]{#1^{\prime}}

\begin{document}

\title[Block Degeneracy]
{Block Degeneracy for Graded Lie Superalgebras of Cartan Type}
\author{Ke Ou}
\address{Department of Statistics and Mathematics, Yunnan University of Finance and Economics,
Kunming, 650221, China.}\email{justinok1311@hotmail.com}

\begin{abstract}
	Let $\fk$ be an algebraically closed field  of characteristic $ p>0. $ In this short note, we illustrate a class of Lie superalgebras over $ \fk $ such that the category of restricted supermodules is of one block. As an application, if $ p>3 $ and $ \lieg $ is a graded restricted Cartan type Lie superalgebra of type W, S and H, then the category of restricted supermodules of $ \lieg $ is of one block.
\end{abstract}

\maketitle
\section{Introduction }
A Lie superalgebra $ \lieg=\liege\ds\liego $ over $ \fk $ is called restricted if $ (\liege,[p]) $ is a restricted Lie algebra with $ p $-mapping $ [p]: \lieg_{\bar{0}}\ra \lieg_{\bar{0}} $  and $ \liego $ is a restricted $ \liege $ module via the adjoint action (cf. \cite{Pe}).
Let $ (\lieg,[p]) $  be a restricted Lie superalgebra and $ U(\lieg) $ be the enveloping superalgebra of $ \lieg. $ One can define the so-called restricted enveloping superalgebra $ u(\lieg)=U(\lieg)/I_p $ where $ I_p $ is the $ \fz_2 $-graded two-sided ideal generated by $ \{ x^p-x^{[p]}\mid x\in \liege \}. $ A $ \lieg $ supermodule $ (V=V_{\bar{0}}\ds V_{\bar{1}},\rho) $ is called restricted if $ \rho $ satisfies $ \rho(x^{[p]})=\rho(x)^p $ for all $ x\in\liege. $ All restricted $ \lieg $-supermodules constitute a full subcategory of the $ \lieg $-supermodule category which coincide with the $ u(\lieg) $-supermodule category denoted by $ u(\lieg)$-\text{smod}. We call $ u(\lieg) $ is of one block if $ u(\lieg) $-smod is of one block.

Over the past decades, the study of modular representations of restricted Lie (super)algebras in prime characteristic has made significant progress (see \cite{HN,Na,YS1,YS2,Zh} for examples).
When $ \lieg=W(0,n) $ over $ \fc, $  Shomron proves in \cite{Sh} that the category of finite-dimensional representations decomposes into blocks parametrized by $ (\fc/\fz)\times\fz_2 $. In contrast to complex case, if either $ \lieg=X(m,1) $ is a Cartan type Lie algebra where $  X\in\{W,S,H,K\} $ (\cite{HN}) or $ \lieg=W(0,n,1) $ is a Cartan type Lie superalgebra (\cite{YS1}) over $ \fk $, the category of restricted (super)modules has only one block. In this paper, we generalize this degeneracy phenomenon of restricted supermodules to the so-called restricted Cartan type Lie superalgebras $ X(m,n,1) $ where $ X\in\{ W,S,H \}. $ 

Our paper is organized as follows. In section 2, we illustrate a class of Lie superalgebras over $ \fk $ such that the category of restricted supermodules is of one block. Section 3 is concerned with the structure of the Cartan type Lie superalgebras. Applying the results in section 2, we obtain the following main theorem in section 4:

\begin{theorem}$ \on{(see\ Theorem\ \ref{main thm})} $
	Let $ \fk $ be an algebraically closed field with characteristics $ p > 3, $ and $ \lieg=X(m,n,1),\ X\in\{ W,S,H \}, $ be a graded restricted Lie superalgebra of Cartan type over $ \fk $ except if $ X=H $ with $ n=4. $
	
	Then $ u(\lieg) $ is of one block.
\end{theorem}

As I know, F.Duan, B.Shu and Y.Yao obtain similar results in \cite{DSY} by a different method.

Entire the whole paper, denote $ I=\{ 0,1,\cdots,p-1 \}. $

\section{Restricted Lie superalgebras with triangular Decomposition}
Let $ \lieg=\liege\ds\liego $ be a restricted Lie superalgebra which admits a triangular decomposition relative to a maximal torus $ \lieh $ of $ \liege: $
\[ \lieg=\liego^-\ds\lien^-\ds \lieh\ds\lien^+ \ds \liego^+ \]
where $ \liege=\lien^-\ds \lieh\ds\lien^+. $ Recall that $ \lien^{\pm} $ are $ p $-nilpotent subalgebras. Set $ \lieb_{\lieg}^{\pm}=\liego^{\pm}\ds\lien^{\pm} \ds \lieh $ and $\lieb_{\liege}^{\pm}=\lien^{\pm} \ds \lieh. $ Analogue to \cite{HN}, this decomposition for $ \lieg $ is long if 
\[ \text{dim}_{\fk}(\lien^- )<\text{dim}_{\fk}(\lien^+) \text{\ and\ } \text{dim}_{\fk}(\liego^-)<\text{dim}_{\fk}(\liego^+).\]

By \cite{Zh}, the iso-classes of simple restricted $ \lieg $ modules are parametrized by restricted weights $ \Lambda=\{ \lambda\in\lieh^*\mid \lambda(h^{[p]}) =\lambda(h)^p,\ \forall h\in\lieh \}. $ If $ \text{dim}(\lieh)=n, $ then $ \Lambda\iso I^n=\{ \lambda= (\lambda_1,\cdots, \lambda_n)\mid \lambda_i\in I, i=1,\cdots,n \}. $  More precisely, for a given $ \lambda\in\Lambda, $ there is a one-dimensional restricted $ \lieb_{\lieg}^+ $ module $ \fk_\lambda=\fk\cdot 1_{\lambda} $ on which $ \lieh $ acts as a scalar determined by $ \lambda $ while $ \liego^+\ds\lien^+ $ acts trivially. Then one has the so-called baby Verma module
\[ V^+(\lambda):=u(\lieg)\ten_{u(\lieb_{\lieg}^+)}\fk_\lambda \]
with simple head $ L(\lambda). $ Moreover,
For any restricted simple module $ \mf{m}, $ there is a $ \lambda\in\Lambda, $ such that $ V^+(\lambda)\sur \mf{m} $ (cf. \cite{Zh}).

Note that $ \lieb_{\lieg}^- $ also satisfies the conditions of \cite[lemma 2.2]{Zh}, then for each $ \lambda\in\Lambda, $ the one-dimensional $ u(\lieb_{\lieg}^-) $ module $ \fk_\lambda $ induces an $ u(\lieg) $ module $$  V^-(\lambda):=u(\lieg)\ten_{u(\lieb_{\lieg}^-)}\fk_\lambda, $$ 
which is indecomposable with simple head. 

For $ M\in u(\lieg) $-smod, let $ [M] $ denote the formal sum of composition factors in the Grothendick ring of $ u(\lieg) $-smod.

\begin{lemma}\label{1}
	Let $ \liel=\liel_{\bar{0}}\ds\liel_{\bar{1}} $ be a restricted Lie superalgebra which admits a triangular decomposition relative to a maximal torus $ \lieh $ of $ \liel_{\bar{0}}: $
	\[ \liel=\liel_{\bar{1}}^-\ds\lien^-_{\liel}\ds \lieh\ds\lien^+_{\liel} \ds \liel_{\bar{1}}^+ \]
	where $ \liel_{\bar{0}}=\lien^-_{\liel}\ds \lieh\ds\lien^+_{\liel}. $
	
	Assume the following:
	\begin{enumerate}
		\item $ \liel_{\bar{1}}^-\ds\lien^-_{\liel}\ds\lien^+_{\liel} \ds \liel_{\bar{1}}^+ $ is a $ p $-nilpotent $ \fz_2 $-graded ideal.
		\item $ \lien^+_{\liel} $ contains $ \textup{dim}(\lieh) $ linear independent weight vectors having linearly independent weights in $ \Lambda. $ 
	\end{enumerate}
	Then for each $ \lambda\in\Lambda, $ $ [V^-(\lambda)] $ is independent of $ \lambda $ and $$ [V^-(\lambda)]=\sum_{\mu\in\Lambda} p^s2^{t}[\fk_\mu], $$
	where $ s=\textup{dim}(\lien^+_{\liel})-\textup{dim}(\lieh),\ t=\textup{dim}(\liel_{\bar{1}}^+) $ and $ \fk_\mu $ is the one dimensional simple $ u(\liel) $ module of weight $ \mu. $
\end{lemma}
\begin{proof}
	By (1), 
	$ \text{rad}(\liel) =\liel_{\bar{1}}^-\ds\lien^-\ds\lien^+ \ds \liel_{\bar{1}}^+. $ Since $ \text{rad}(\liel) $ is $ p $-nilpotent and finite dimension, each restricted representations of $ \liel $ is one dimension \cite[lemma 2.2]{Zh}. Let $ \{ \fk_\mu\mid \mu\in\Lambda \} $ represent the set of non-isomorphic simple $ u(\liel) $ modules.
	
	The composition factors of a module can be obtained by computing its weight spaces. From (2), suppose $ n=\text{dim}(\lieh), $ $ \liel_{\bar{1}}^+ $ has  basis $ \{ z_1,\cdots,z_t \} $ and  $ \lien^+ $ has  basis $ \{ x_1,\cdots, x_n,y_1,\cdots, y_s \} $ where $ x_i$ is of weight $ \alpha_i\in\Lambda $ for each $ i=1,\cdots,n $ such that $ \alpha_1,\cdots, \alpha_n $ are linear independent. Then $ x_1^{i_1}\cdots x_n^{i_n} $ has weight $ i_1\alpha_1+\cdots +i_n\alpha_n $ for each $ (i_1, \cdots, i_n)\in I^n. $
	
	For each choice of $ \un{j}\in I^s $ and $ \un{k}\in B(t), $ as $ u(\lieh) $ module, $$ N=\textup{span}_{\fk} \{  X^{\un{i}}Y^{\un{j}}Z^{\un{k}}\mid \un{i}\in I^n \} $$ must have all weights occurring with multiplicity 1 . Since 
	$$ V^-(\lambda)=\textup{span}_{\fk}\{ X^{\un{i}}Y^{\un{j}}Z^{\un{k}}\ten  1_{\lambda}\mid \un{i}\in I^n, \un{j}\in I^s  \text{ and } \un{k}\in B(t) \}, $$ then all possible weights occuring with the same multiplicity $ p^s2^t $ in $ V^-(\lambda). $   
	
	Namely, $ [V^-(\lambda)]=\sum_{\mu\in\Lambda} p^s2^{t}[\fk_\mu] $ which is independent of $ \lambda. $
\end{proof}

\begin{proposition}
	Let $ \lieg=\lieg_{\bar{0}}\ds\lieg_{\bar{1}} $ be a restricted Lie superalgebra which admits a long triangular decomposition relative to a maximal torus $ \lieh $ of $ \lieg_{\bar{0}}: $
	\[ \lieg=\lieg_{\bar{1}}^-\ds\lien^-\ds \lieh\ds\lien^+ \ds \lieg_{\bar{1}}^+,\ \lieg_{\bar{0}}=\lien^-\ds \lieh\ds\lien^+. \]
	Assume the following:
	\begin{enumerate}
		\item $ \lieg $ has a restricted subalgebra $ \liel $ satisfies the assumptions of lemma \ref{1}.
		\item $ \lieb^-_{\liel}=\lieb_{\lieg}^-. $
		\item $ \lien^-_{\liel}=\lien^- $ has at least $ \on{dim}(\lieh) $ linearly independent vectors having linearly independent weights in $ \Lambda. $
	\end{enumerate}
	Then for each $ \lambda\in\Lambda, $ $$ [V^-(\lambda)]=\sum_{\mu\in\Lambda} p^{s}2^t[V^+(\mu)], $$ where $ s=\textup{dim}(\lien^+)-\textup{dim}(\lien^-)-\textup{dim}(\lieh),$  $ t= \textup{dim}(\liego^+)-\textup{dim}(\liego^-). $ 
\end{proposition}
\begin{proof}
	By lemma 3.1, for each $ \lambda\in\Lambda, $ 
	\[ [V^-(\lambda)]=[u(\lieg)\ten_{u(\liel)}[u(\liel)\ten_{u(\lieb_{\liel}^-)} \fk_\lambda ] ]=\sum_{\mu\in\Lambda}p^\alpha 2^\beta[ u(\lieg)\ten_{u(\liel)} \fk_\mu ], \] where $ \alpha= \textup{dim} (\liel)-\dim(\lien^-),\ \beta=\dim(\liego^+). $ 
	
	In particular, $ [V^-(\lambda)] $ is independent of $ \lambda. $
	
	By assumption (3), $ \lieb_{\lieg}^{\pm} $ satisfies the assumptions of lemma 3.1. Therefore,
	\[ [u(\lieg)\ten_{u(\lieh)} \fk_\lambda]= [u(\lieg)\ten_{u(\lieb_{\lieg}^\pm)}[u(\lieb_{\lieg}^\pm)\ten_{u(\lieh)}\fk_\lambda ] ]=\sum_{\mu\in\Lambda}p^{s_\pm} 2^{t_\pm}[ u(\lieg)\ten_{u(\lieb_{\lieg}^\pm)} \fk_\mu ], \]
	where $ s_\pm=\dim (\lien^\pm),\ t_\pm=\dim(\liego^\pm). $
	Since the triangular decomposition is long, i.e. $ s_+>s_-\textup{ and } t_+>t_-, $ we have
	\[ \sum_{\mu\in\Lambda}[ u(\lieg)\ten_{u(\lieb_{\lieg}^-)}\fk_\mu ] = \sum_{\mu\in\Lambda}p^{s_+-s_-}2^{t_+-t_-}[u(\lieg)\ten_{u(\lieb_{\lieg}^+)}\fk_\mu]. \] Note that $ [V^-(\lambda)] $ is independent of $ \lambda. $ Therefore, for all $ \lambda\in\Lambda, $
	\[ p^{\dim(\lieh)}[V^-(\lambda)]=\sum_{\mu\in\Lambda} p^{s_+-s_-}2^{t_+-t_-}[V^+(\mu)], \]
	\[ [V^-(\lambda)]=\sum_{\mu\in\Lambda} p^{s_+-s_--\dim(\lieh)} 2^{t_+-t_-}[V^+(\mu)]. \]
\end{proof}

\begin{proposition}
	Let $ \lieg=\lieg_{\bar{0}}\ds\lieg_{\bar{1}} $ be a restricted Lie superalgebra which admits a triangular decomposition relative to a maximal torus $ \lieh $ of $ \lieg_{\bar{0}}: $
	\[ \lieg=\lieg_{\bar{1}}^-\ds\lien^-\ds \lieh\ds\lien^+ \ds \lieg_{\bar{1}}^+,\ \lieg_{\bar{0}}=\lien^-\ds \lieh\ds\lien^+. \]
	Assume that there exists a subalgebra $ \liel $ such that:
	\begin{enumerate}
		\item $ \lieb^-_{\liel}=\lieb_{\lieg}^-. $
		\item $ \liel $ is a classical Lie superalgebra and there is a bijection $ \psi:\Lambda\ra\Lambda, $ such that
		$$ [u(\liel)\ten_{u(\lieb_{\liel}^-)}\fk_\lambda] = [u(\liel)\ten_{u(\lieb_{\liel}^+)}\fk_{\psi(\lambda)}]. $$
		\item A vector space complementary to $ \liel $, in $ \lieg, $ has at least $ \on{dim}(\lieh) $ linearly independent vectors having linearly independent weights in $ \Lambda. $
	\end{enumerate}
	Then $$ [V^-(\lambda)]=\sum_{\mu\in\Lambda} p^{s}2^t[V^+(\mu)], $$ where $ s=\textup{dim}(\lien^+)-\textup{dim}(\lien^-)-\textup{dim}(\lieh), t= \textup{dim}(\liego^+)-\textup{dim}(\liego^-). $ 
\end{proposition}
\begin{proof}
	Similar to the proof of lemma \ref{1}, for all $ \lambda\in\Lambda, $ we have
	\[  [ u(\lieb_{\lieg}^+) \ten_{u(\lieb_{\liel}^+)} \fk_{\lambda}  ] = \sum_{\mu\in\Lambda}p^s 2^t [\fk_\mu],  \]
	where $ s $ and $ t $ are defined in proposition.
	
	By assumption (1) and (2), we have
	
	\begin{tabular}{rcl}
		$ [ u(\lieg)\ten_{u(\lieb_\lieg^-)}\fk_\lambda ]$ &=& $[ u(\lieg)\ten_{u(\liel)} (u(\liel) \ten_{u(\lieb_{\liel}^-)} \fk_\lambda) ]$	\\
		&=& $[ u(\lieg)\ten_{u(\liel)} (u(\liel) \ten_{u(\lieb_{\liel}^+)} \fk_{\psi(\lambda)} ) ] $	\\
		&=& $ [ u(\lieg)\ten_{u(\lieb_{\lieg}^+)} (u(\lieb_{\lieg}^+) \ten_{u(\lieb_{\liel}^+)} \fk_{\psi(\lambda)} ) ] $\\
		&=& $ \sum_{\mu\in\Lambda}p^s2^t[u(\lieg)\ten_{u(\lieb_\lieg^+)}\fk_\mu]. $
	\end{tabular}
	
	Proposition holds.
\end{proof}
\begin{remark}
	The Lie algebra version of lemma 2.1 and proposition 2.2 (resp. proposition 2.3) are investigated in \cite{HN} (resp. \cite{Na}).
\end{remark}

\begin{corollary}
	If $ \lieg $ is a restricted Lie superalgebra satisfies all assumptions in proposition 2.2 or 2.3, then $ u(\lieg) $ is of one block.
\end{corollary}
\begin{proof}
	Note that $ V^-(\lambda) $ is indecomposable with simple head for all $ \lambda\in\Lambda. $ The proof of corollary 2.4 in \cite{HN} still works. Hence, the corollary holds.
\end{proof}

\section{Restricted Cartan Type Lie Superalgebras}

For given positive integers $ m $ and $ n, $ put
\[ \tau{(i)}=\left \{ \begin{tabular}{ll}
$ 0, $ & $ 1\leq i\leq m; $\\
$ 1, $ & $ m+1\leq i\leq m+n. $
\end{tabular}\right . \]
For $ 0\leq k\leq n, $ set $ \mathbb{B}_k=\{  (i_1,\cdots,i_k) \mid m+1\leq i_1<\cdots<i_k\leq m+n  \} $ and $ \bb{B}(n)=\cup_{i=0}^n\bb{B}_k $ where $ \bb{B}_0=\emptyset. $ 

Let $ A(m,1) $ denote the truncated divided power algebra over $ \fk $ with a basis $ \{ x^{(\alpha)}\mid \alpha\in I^m \}. $ For $ \epsilon_i=(\delta_{i1}, \cdots, \delta_{im}), $ we abbreviate $ x^{(\epsilon_i)} $ to $ x_i,\ i=1,\cdots,m. $
Let $ \Lambda(n) $ be the Grassmann superalgebra over $ \fk $ in $ n $ variables $ x_{m+1},\cdots,x_{m+n} $ with basis $ \{ x^{(\beta)}\mid \beta\in \bb{B}(n) \} $ where $ x^{(\beta)}=x_{i_1}\cdots x_{i_k} $ if $ \beta= (i_1,\cdots, i_k). $ Denote the tensor product by $ A(m,n,1)=A(m,1)\ten \Lambda(n). $ Then $ A(m,n,1) $ is an associative superalgebra with a $ \fz_2 $-gradation induced by the trivial $ \fz_2 $-gradation of $ A(m,1) $ and the natural $ \fz_2 $-gradation of $ \Lambda(n). $ Denote $ \text{d}(f) $ the parity of $ f\in A(m,n,1). $

Let $ D_1,\cdots,D_{m+n} $ be the superderivations of the superalgebra $ A(m,n,1) $ such that $ D_i(x_j)=\delta_{ij} $ for $ 1\leq i,j\leq m+n. $ Define
$  W(m,n,1)=\left\{ \sum_{i=1}^{m+n} f_iD_i\mid f_i\in A(m,n,1),1\leq i\leq m+n \right\}. $ 

Then $ W(m,n,1) $ is a restricted Lie superalgebra of Witt type.  The $ \fz $-grading of $$ W(m,n,1)=\ds_{i\in\fz}W(m,n,1)_i $$ is induced by $ \abs{x_i}=1 $ and $ \abs{D_i}=-1 $ for all $ 1\leq i\leq m+n. $ Namely,
\[ W(m,n,1)_i=\left\{ \sum_{j=1}^{m+n} f_jD_j\mid \abs{f_j}=i+1 \right\}. \]

For each pair $ 1\leq i,j\leq m+n $ defines $ D_{ij}:A(m,n,1)\ra W(m,n,1) $ by $$ D_{ij}(f)=f_iD_i+f_jD_j $$ where $ f $ is homogeneous and $$ f_i=-(-1)^{ \text{d}(f)(\tau(i)+\tau(j))} D_j(f),\  f_j=(-1)^{\tau(i)\tau(j)}D_i(f). $$ 

The special superalgebra $ S(m,n,1) $ is defined by \[ S(m,n,1)=\langle D_{ij}(f)\mid f \text{ is homogeneous},\ 1\leq i<j\leq m+n \rangle. \]

$ S(m,n,1) $ is a $ \fz $-graded restricted subalgebra of $ W(m,n,1). $ The $ \fz $-grading structure is given by $ S(m,n,1)_i:=S(m,n,1)\cap W(m,n,1)_i. $ 

Next we define the Hamiltonian type Lie superalgebra $ H(m,n,1), $ where $ m=2l $ is even and $ n>3. $ Let
\[ \pr{i}=\left \{ \begin{tabular}{ll}
$ i+l, $ & $ 1\leq i\leq l, $\\
$ i-l, $ & $ l+1\leq i\leq m, $\\
$ i, $ & $ m<i\leq m+n; $
\end{tabular}\right . 
 \sigma(i)=\left \{ \begin{tabular}{ll}
$ 1, $ & $ 1\leq i\leq l, $\\
$ -1, $ & $ l+1\leq i\leq m, $\\
$ 1, $ & $ m<i\leq m+n. $
\end{tabular}\right .
\]

The Hamiltonian operator $ D_H $ is defined as follows:
\[ \begin{tabular}{lccl}
$ D_H: $ & $ A(m,n,1)$ & $\ra $ & $ W(m,n,1) $\\
&$ f$ & $\mt $ & $ D_H(f)=\sum_{i=1}^{m+n}f_iD_i $
\end{tabular} \]
where $ f $ is homogeneous and $ f_i=\sigma(\pr{i})(-1)^{\tau(\pr{i}) \text{d}(f)}D_{\pr{i}}(f). $

The Hamiltonian superalgebra $ H(m,n,1) $ is defined by
\[ \bar{H}(m,n,1)=\langle D_H(f)\mid f \text{ is homogeneous} \rangle, \]
\[ H(m,n,1)=[\bar{H}(m,n,1),\bar{H}(m,n,1)]. \]

$ H(m,n,1) $ is a $ \fz $-graded restricted subalgebra of $ W(m,n,1).$ The $ \fz $-grading structure is given by $ H(m,n,1)_i:=H(m,n,1)\cap W(m,n,1)_i. $

\section{Blocks of Cartan Type Lie Superalgebra}
Entire this section, assume $ p>3. $
\subsection{Type W}

For $ W(m,n,1), $ there is no subalgebra $ \liel $ satisfying the hypothesis of proposition 3.2 (in fact, assumption (3) fails). Hence, we need proposition 3.3.

Let $ \liel $ be a restricted Lie superalgebra of classical type with triangular decomposition $ \liel=\liel^-\ds\lieh\ds\liel^+ $ with respect to $ \lieh. $ Suppose $ \sigma $ is an even restricted automorphism of $ \liel $ 
such that $ \sigma(\lieh)\subseteq \lieh. $ Then it induces $ \tilde{\sigma}:\lieh^* \ra \lieh^* $ by $$ \tilde{\sigma}(\lambda)(h)= -\lambda(\sigma(h)) $$ where $ \lambda\in \lieh^*,\ h\in\lieh. $ Moreover, $ \tilde{\sigma}(\Lambda)\subseteq \Lambda. $ 

Denote $ \lieb=\lieh\ds\liel^+ $ a solvable subalgebra of $ \liel $ and
$  V(\lambda)=u(\liel)\ten_{u(\lieb)}\fk_{\lambda}  $ the baby Verma module
where $ \fk_{\lambda} $ is a one-dimensional $ u(\lieb) $ module with weight $ \lambda. $ Let $ V^\sigma(\lambda) $ be the twisted baby Verma module. Namely, $ V^\sigma(\lambda)\iso V(\lambda) $ as vector spaces while $ x\cdot m:= \sigma(x)(m) $ for all $ x\in u(\liel),\ m\in V^\sigma(\lambda). $ 

The following lemma is a straightforward calculation.

\begin{lemma}
	Keep assumptions as above, $ V^\sigma(\lambda)\iso u(\liel)\ten_ {u(\sigma^{-1}(\lieb))} 1_{-\tilde{\sigma}(\lambda)} $ by sending $ x\ten 1_\lambda $ to $ \sigma^{-1}(x)\ten 1_{-\tilde{\sigma}(\lambda)}. $
	In particular, $ [V(\lambda)] =[V^\sigma(\lambda)]= [u(\liel)\ten_ {u(\sigma^{-1}(\lieb))} 1_{-\tilde{\sigma}(\lambda)}]. $

\end{lemma}

Let $ \lieg=W(m,n,1)=\lieg^-\ds\lieh\ds\lieg^+ $ be the triangular decomposition related to maximal torus $ \lieh=\langle h_1,\cdots,h_{m+n} \rangle $ where $ h_i:= x_iD_i$ for $ i=1,\cdots,m+n. $ 

Now, set 
\[ \liel=\langle D_1,\cdots,D_{m+n} \rangle\ds \lieg_0\ds\langle p_1,\cdots, p_{m+n} \rangle \]
where $ p_i=x_i\sum_{j=1}^{m+n}x_jD_j\in\lieg_1. $ 

Thanks to \cite[lemma 3.1]{CLO}, $ \liel\iso \mathfrak{pgl}(m+1|n). $ Let $e_i:=E_{i,i+1},\ f_i:=E_{i+1,i}. $ Then $ \{ e_i,f_i\mid i=1,\cdots,m+n-1 \}$ generates $ \mathfrak{pgl}(m+1|n). $ There is an even restricted automorphism $ \alpha $ of $ \liel $ induced by $ \alpha(e_i)=f_i $ and $ \alpha(f_i)=e_i. $
Note that $ \alpha(\lieb_{\liel}^\pm)= \lieb_{\liel}^\mp $ and $ \tilde{\alpha} $ keeps $ \Lambda. $ By lemma 4.1, we have
\[  [u(\liel)\ten_ {u(\lieb_{\liel}^-)} 1_{\lambda}]= [u(\liel)\ten_ {u(\lieb_{\liel}^+)} 1_{-\tilde{\alpha}(\lambda)}]. \]

Therefore, $ \liel $ is a subalgebra satisfying (1) and (2) of proposition 3.3.

For each $ i=1,\cdots,m+n,\ x_i^3D_i $ has weight $ 2\gamma_i $ where $ \gamma_{i}\in\Lambda $ such that $ \gamma_i (h_j)=\delta_{ij}. $ Therefore, assumption (3) of proposition 3.3 satisfies.

To sum above up, all assumptions of proposition 3.3 hold for $ W(m,n,1) $ and hence we have the following proposition by corollary 3.4.

\begin{proposition}
	Keep assumptions as above, and let $ \lieg=W(m,n,1), $ then $ u(\lieg) $ is of one block.
\end{proposition}

\subsection{Type S}
Let $ \lieg= S(m,n,1)=\liego^-\ds \lien^-\ds\lieh\ds\lien^+\ds \liego^+ $ be the triangular decomposition related to maximal torus $ \lieh=\langle h_i\mid 1\leq i\leq m+n-1 \rangle $ where $ h_i:= x_iD_i-x_{i+1}D_{i+1}, $ and $ \liege=\lien^-\ds\lieh\ds\lien^+. $

Set $ \liel=\liego^-\ds \lien^-\ds\lieh\ds\lien_1^+\ds\liel_{\bar{1}}^+ $ where
$ \lien_1^+= \mf{s}\ds\mf{t},  $
\[ \mf{s}=\langle x^{(a)}x_{m+i}D_{m+n}\mid (a)\in I^m\xg \{0\},1\leq i<n  \rangle,  \]
\[ \mf{t}=\langle x^{(a)}D_m\mid (a)\in I^m, a_m=0, \abs{a}\geq 2 \rangle,\qquad\quad \]
\[ \liel_{\bar{1}}^+=\langle x^{(a)}D_{m+n}\mid (a)\in I^m\xg\{0\} \rangle,\qquad\qquad\qquad\  \]
\[ \liego^-=\langle D_{m+1},\cdots,D_{m+n} \rangle\ds \langle x_iD_j\mid 1\leq i\leq m< j\leq m+n \rangle,\text{ and }\qquad\quad  \]
\[ \lien^-=\langle D_{1},\cdots,D_{m} \rangle\ds \langle  x_iD_j\mid 1\leq i<j\leq m \text{ or } m+1\leq i< j\leq m+n\rangle. \]

One can check the followings:
\begin{itemize}
	\item $ [\lieh,\liego^-\ds \lien^-\ds\lien_1^+\ds \liel_{\bar{1}}^+]\subseteq \liego^-\ds \lien^-\ds\lien_1^+\ds \liel_{\bar{1}}^+; $
	\item $ \liego^-\ds\lien^- $ is $ p $-nilpotent;
	\item $ [ \lien^-,\lien_1^+\ds\liel_{\bar{1}}^+]\subseteq \lien_1^+\ds \liel_{\bar{1}}^+; $
	$ [\liego^-,\lien_1^+]\subseteq \liel_{\bar{1}}^+,\ [\liego^-,\liel_{\bar{1}}^+]=0; $
	\item $ [ \mf{s},\mf{s} ]=[\mf{t},\mf{t}]=[\mf{s},\liel_{\bar{1}}^+]= [\liel_{\bar{1}}^+, \liel_{\bar{1}}^+]=0,\ [\mf{s},\mf{t}] \subseteq \mf{s}, $ and $ [\mf{t},\liel_{\bar{1}}^+]\subseteq \liel_{\bar{1}}^+. $
\end{itemize}

Therefore, $ \text{rad}(\liel)=\liego^-\ds \lien^-\ds\lien_1^+\ds \liel_{\bar{1}}^+ $ is a $ p $-nilpotent ideal and 
$ \liel $ is a subalgebra satisfying (1) in lemma 3.1.

Now define $ \gamma_{i}\in\Lambda $ by $ \gamma _i(h_j)=\delta_{ij}, $  $ 1\leq i,j\leq m+n-1. $ 

For $ 1\leq i\leq m$ and $ 1\leq j\leq n-1, $ $ D_i $ has weight $ -\gamma_i $ while $ x_{m+j}D_{m+j+1} $ has weight $ -\gamma_{m+j-1}+2\gamma_{m+j} -(1-\delta_{j,n-1})\gamma_{m+j+1} $ with respect to $ \lieh. $ 
One can check that  $ \lien^- $ contains $ m+n-1 $ linear independent vectors $$\{ D_i,\ x_{m+j}D_{m+j+1}\mid 1\leq i\leq m;\ 1\leq j\leq n-1\}. $$ 
with linear independent weights. Therefore, assumption (2) and (3) of proposition 3.2 satisfy.

For $ 1\leq i\leq m-1,\ x_i^2D_m $ has weight $ 2\gamma_i-2(1-\delta_{1,i}) \gamma_{i-1}+\gamma_{m-1}-\gamma_{m} $ while $ x_1^3D_m $ has weight $ 3\gamma_1+\gamma_{m-1}-\gamma_{m}. $ 

For $ 1\leq j\leq n-1,\ x_1x_{m+j}D_{m+n} $ has weight $ \gamma_1+\gamma_{m+j-1}-\gamma_{m+j}-\gamma _{m+n-1}. $ 

Hence, $ \lien^+_1 $ contains $ m+n-1 $ linear independent vectors $$\{ x_1^2D_i,\ x_1x_{m+j}D_{m+n}\mid 2\leq i\leq m;\ 1\leq j\leq n-1\}\cup\{x_1^3D_2\} $$ with linear independent weights. Assumption (2) of lemma 3.1 satisfies.

To sum above up, all assumptions of proposition 3.2 hold for $ S(m,n,1) $ and hence we have the following proposition by corollary 3.4.

\begin{proposition}
	Keep assumptions as above, and let $ \lieg=S(m,n,1), $ then $ u(\lieg) $ is of one block.
\end{proposition}

\subsection{Type H}
Let $ \lieg:=H(m,n,1),$ where $ m=2l,\ n>3.$ Denote $ k=[n/2]. $ For every $ (a)=(a_1,\cdots,a_m)\in I^m $ and $ (b)=(b_1,\cdots,b_u)\in\bb{B}_u\subseteq\bb{B}(n), $ denote $$ X^{(a)}Y^{(b)}= x_1^{a_1}\cdots x_m^{a_m}x_{m+b_1}\cdots x_{m+b_u}. $$ By definition, 
$\lieg=\langle D_H(X^{(a)}Y^{(b)})\mid 
X^{(a)}Y^{(b)}\neq x_1^{p-1} \cdots x_m^{p-1}x_{m+1}\cdots x_{m+n}\rangle. $ 

Fix a maximal torus $ \lieh $ with basis $$ \{ h_i,\ h_{m+j} \mid i=1,\cdots,l;j=1,\cdots,k \} $$
where $ h_i=D_H(x_ix_{l+i}),\  h_{m+j}= D_H(\sqrt{-1}x_{m+j}x_{m+k+j}). $

For $ 1\leq i\leq k, $ set 
$ e_i:=x_{m+i}+\sqrt{-1}x_{m+k+i}\text{ and } f_i:=x_{m+i}-\sqrt{-1}x_{m+k+i}. $ Then both $ D_H(e_i)$ and $D_H(f_i) $ are homogeneous odd elements of degree $ -1 $.
 
Define the following subspaces of $ \lieg: $

\begin{tabular}{ccl}
$ \alpha $ &=& $\langle D_i\mid m+1\leq i\leq m+n \rangle \ds \langle D_H(x_if_j)\mid 1\leq i\leq m; 1\leq j\leq k \rangle; $\\
$ \beta$ &=& $ \langle D_H(x_ix_{m+n})\mid l+1\leq i\leq m \rangle; $\\
$ \lien_1^-$ &=& $\langle D_H(x_ix_j)\mid l+1\leq i, j\leq 2l\text{ or }1\leq i< l\leq j<i+l \rangle  $  $\ds \langle D_1,\cdots,D_m \rangle ;$\\
$ \lien_2^{\prime}$ &=& $\langle D_H(a_{ij}),\ D_H(b_{ij})\mid 1\leq i<j\leq k \rangle $ where $ a_{ij}=f_ie_j,$  $ b_{ij}=f_if_j; $\\
$ \lien_3^{\prime}$ &=& $\langle D_H(f_{i}x_{m+n})\mid 1\leq i\leq k \rangle.  $
\end{tabular}

Suppose $ \lieg=\liego^-\ds \lien^-\ds\lieh\ds\lien^+\ds \liego^+ $ be the triangular decomposition related to maximal torus $ \lieh $ and $ \liege=\lien^-\ds\lieh\ds\lien^+. $ 

One can check that 
$ \lien^-=\lien_1^-\ds\lien_2^- $ where
\[ \lien_2^-=\left \{ \begin{tabular}{ll}
$ \pr{\lien_2}, $ & $ n=2k, $\\
$ \pr{\lien_2}\ds\pr{\lien_3}, $ & $ n=2k+1; $
\end{tabular}\right . 
\text{and }\ 
\liego^-=\left \{ \begin{tabular}{ll}
$ \alpha, $ & $ n=2k, $\\
$ \alpha\ds\beta, $ & $ n=2k+1. $
\end{tabular}\right . 
\]

\begin{remark}
	Above description for $ \lien_2^- $ comes from \cite[section 2.2]{YS2}.
\end{remark}

For each $ (c)=(c_1,\cdots,c_u)\in\bb{B}(k),\ 0\leq u\leq k, $ denote $ f^{(c)}=f_{c_1}\cdots f_{c_u}. $ In this case, the parity of $ D_H(f^{(c)}) $ equals to the parity of $ u=\abs{(c)}. $ 

Now, if $ n=2k $ is even, define 
$ \liel^+ =\langle D_H(x^{(a)} f^{(c)})\mid  a_j=0\text{ if } 1\leq j\leq l; \abs{(a)}+\abs{(c)} \geq 3 \rangle. $

If $ n=2k+1 $ is odd, define 
$ \liel^+ =\langle D_H(x^{(a)} f^{(c)}x_{m+n}^{\delta})\mid  a_j=0\text{ if } 1\leq j\leq l; \delta\in\{0,1\} ;\abs{(a)}+\abs{(c)}+\delta \geq 3 \rangle. $

Let $ \lien_1^+ $ (resp. $ \liel_{\bar{1}}^+ $) be the even (resp. odd) part of $ \liel^+. $ 

One can check that $  \liel:=\liego^-\ds\lien^-\ds\lieh\ds \lien_1^+\ds \liel_{\bar{1}}^+ $ is a restricted subalgebra of $ \lieg. $
Note that for all $ f,g\in A(m,n,1) $ homogeneous,
$$ [D_i, D_H(f)]=D_H(D_i(f)) \text{ and}$$
\[ [D_H(f),D_H(g)]=D_H\left(\sum_{i=1}^{m+n}\sigma(i)(-1)^{\tau(i)\text{d}(f)}D_i(f) D_{\pr{i}}(g)\right). \] 

We have the followings: 
\begin{itemize}
	\item $ [D_H(e_i),D_H(e_j)]=[D_H(f_i),D_H(f_j)]=0; $
	\item $ [D_H(e_i),D_H(f_j)]=[D_H(f_j),D_H(e_i)]=-2\delta_{ij}; $
	\item $ [\lieh,\liego^-\ds \lien^-\ds\lien_1^+\ds \liel_{\bar{1}}^+]\subseteq \liego^-\ds \lien^-\ds\lien_1^+\ds \liel_{\bar{1}}^+; $
	\item $ \liego^-\ds\lien^- $ is $ p $-nilpotent;
	\item $ [ \lien^-,\lien_1^+\ds\liel_{\bar{1}}^+]\subseteq \lien_1^+\ds \liel_{\bar{1}}^+; $
	$ [\liego^-,\lien_1^+]\subseteq \liel_{\bar{1}}^+,\ [\liego^-,\liel_{\bar{1}}^+]=0; $
	\item $ [ \lien_1^+\ds\liel_{\bar{1}}^+,\lien_1^+\ds\liel_{\bar{1}}^+ ]=0. $
\end{itemize}

Therefore, $ \text{rad}(\liel)=\liego^-\ds \lien^-\ds\lien_1^+\ds \liel_{\bar{1}}^+ $ is a $ p $-nilpotent ideal and 
$ \liel $ is a subalgebra satisfying (1) in lemma 3.1.

For each $1\leq i,j\leq l,\ 1\leq u,v\leq k, $ defines $ \gamma_{i},\delta_j\in \Lambda $ by $ \gamma_{i}(h_j)=\delta_{ij}, $ and $\delta_u(h_{m+v})=\delta_{uv}. $  

For $ 1\leq i\leq l$ and $ 1\leq j\leq k, $ $ D_i $ has weight $ -\gamma_i $ while $ D_H(e_j) $ (resp. $ D_H(f_j) $) has weight $ \delta_j $ (resp. $ -\delta_{j} $) with respect to $ \lieh. $ Then $ b_{uv} $ has weight $ -\delta_u-\delta_{v} $ for each $ 1\leq u,v\leq k $
and $ a_{12} $ has weight $ -\delta_1+\delta_2. $ 

Therefore, $ \lien^- $ contains $ l+k $ linear independent vectors $$\{ D_i,\ D_H(b_{j,j+1})\mid 1\leq i\leq l;\ 1\leq j\leq k-1\}\cup \{ D_H(a_{12}) \}. $$
with linear independent weights. Assumption (2) and (3) of proposition 3.2 satisfies.

For $ 1\leq i\leq l,\ D_H(x_{l+i}^3) $ has weight $ 3\gamma_i. $  For $ 1\leq u<v\leq k-1,\ D_H(x_{l+1}^2b_{u,v}) $ has weight $ 2\gamma_1-\delta_u-\delta_ {v}, $ and $ D_H(x_{l+1}^2f_kx_{m+n}) $ has weight $ 2\gamma_1-\delta_k $ if $ n=2k+1. $ Moreover,
$ \lien^+ $ contains $ l+k $ linear independent vectors $ S $  with linear independent weights as following:
	
	If $ k\geq 3 $ is odd, $$S=\{ D_H(x_{l+i}^3), D_H(x_{l +1}^2 b_{j,j+1})\mid 1\leq i\leq l; 1\leq j\leq k-1\}\cup \{ D_H(x_{l+1}^2 b_{1k}) \}; $$
		
	If $ k\geq 3 $ is even, $$S=\{ D_H(x_{l+i}^3), D_H(x_{l +1}^2 b_{j,j+1})\mid 1\leq i\leq l; 1\leq j\leq k-1\}\cup \{ D_H(x_{l+1}^2 b_{2k}) \}; $$
	
	If $ n=5, $ $$S=\{ D_H(x_{l+i}^3)\mid 1\leq i\leq l\}\cup \{ D_H(x_{l +1}^2 b_{12}),D_H(x_{l+1}^2 f_{12}x_{m+5}) \}. $$

Therefore, assumption (2) of lemma 3.1 satisfies if $ n>4. $

\begin{proposition}
	Keep assumptions as above, and let $ \lieg=H(m,n,1)$ with $ n>4, $ then $ u(\lieg) $ is of one block.
\end{proposition}

By Proposition 4.2, 4.3 and 4.4, we have our main theorem as follows.

\begin{theorem}\label{main thm}
	Let $ \fk $ be an algebraically closed field with characteristics $ p > 3, $ and $ \lieg=X(m,n,1),\ X\in\{ W,S,H \}, $ be a graded restricted Lie superalgebra of Cartan type over $ \fk $ except if $ X=H $ with $ n=4. $
	
	Then $ u(\lieg) $ is of one block.
\end{theorem}



\end{document}